\documentclass[11pt]{article}
\usepackage{amsmath,amsthm,amssymb}

\usepackage{tikz}
\usepackage{tikzpagenodes} 
 \usetikzlibrary{calc}

\def\titlerunning#1{\gdef\titrun{#1}}
\makeatletter
\def\author#1{\gdef\autrun{\def\and{\unskip, }#1}\gdef\@author{#1}}
\def\address#1{{\def\and{\\\hspace*{18pt}}\renewcommand{\thefootnote}{}%
\footnote {#1}}% 
\markboth{\autrun}{\titrun}}
\makeatother
\def\email#1{\hspace*{4pt}{\em e-mail}: #1}
\def\MSC#1{{\renewcommand{\thefootnote}{}%
\footnote{\emph{Mathematics Subject Classification (2010):} #1}}}
\def\keywords#1{\par\medskip
\noindent\textbf{Keywords:} #1}

\newtheorem{theorem}{Theorem}[section]
\newtheorem{prop}[theorem]{Proposition}

\newtheorem{lemma}[theorem]{Lemma}

\theoremstyle{definition}

\newtheorem{remark}[theorem]{Remark}

\numberwithin{equation}{section}

\def\cL{\mathcal L}

\def\cA{\mathcal A}
\def\cC{\mathcal C}

\def\cF{\mathcal F}

\def\cH{\mathcal H}

\def\cK{\mathcal K}

\def\cP{\mathcal P}

\def\cT{\mathcal T}
\def\cR{\mathcal R}
\def\cS{\mathcal S}
\def\cW{\mathcal W}
\def\cQ{\mathcal Q}
\def\PG{{\rm PG}}

\def\GF{{\rm GF}}

\def\PGL{{\rm PGL}}

\def\GL{{\rm GL}}

\begin{document}

%%%%% To ease editing, add:

\baselineskip=16pt

%%%%%%%%%%%%%%%%
%% In the running head, give an abbreviation of the title. 
\titlerunning{}

\title{Optimal subspace codes in $\PG(4,q)$}

\author{Antonio Cossidente
\and
Francesco Pavese 
\and 
Leo Storme}

\date{}

\maketitle

\address{A. Cossidente: Dipartimento di Matematica, Informatica ed Economia, Universit\`a degli Studi della  Basilicata, Contrada Macchia Romana, 85100 Potenza, Italy;  \email{antonio.cossidente@unibas.it}
\and 
F. Pavese: Dipartimento di Meccanica, Matematica  e Management, Politecnico di Bari, Via Orabona 4, 70125 Bari, Italy; \email{francesco.pavese@poliba.it}
\and
L. Storme: Department of Mathematics, Ghent University, Krijgslaan 281, 9000 Ghent, Belgium; \email{leo.storme@ugent.be}
}

\bigskip

\MSC{Primary 05B25; Secondary 51E20, 94B25}

%%%%%%%%

\begin{abstract}
We investigate subspace codes whose codewords are subspaces of $\PG(4,q)$ having non--constant dimension. In particular, examples of optimal mixed--dimension subspace codes are provided, showing that $\cA_q(5,3) = 2(q^3+1)$.

\end{abstract}

\section{Introduction}

Let $V$ be an $n$--dimensional vector space over $\GF(q)$, $q$ any prime power. The set $S(V)$ of all subspaces of $V$, or subspaces of the projective space $\PG(V)$, forms a metric space with respect to the {\em subspace distance} defined by $d_s(U,U')=\dim (U+U')- \dim(U\cap U')$. 
In the context of subspace codes, the main problem is to determine the largest possible size of codes in the space $(S(V),d_s)$ with a given minimum distance, and to classify the corresponding optimal codes. The interest in these codes is a consequence of
the fact that codes in the projective space and codes in the Grassmannian over a finite field referred to as subspace codes and constant--dimension codes (CDC), respectively, have been proposed for error control in random linear network coding.
An $(n,M,d)_q$ mixed--dimension subspace code is a set $\cC$ of subspaces of $V$ with $|\cC| = M$ and minimum subspace distance $d_s(\cC)=\min\{d_s(U,U') \;\; | \;\; U,U'\in \cC, U \ne U' \}=d$. The maximum size of an $(n,M,d)_q$ mixed--dimension subspace code is denoted by $\cA_q(n,d)$. In this paper we discuss the smallest open mixed--dimension case, which occurs when $n=5$ and $d=3$. In particular, examples of optimal mixed--dimension subspace codes are provided, showing that $\cA_q(5,3) = 2(q^3+1)$.

We wish to remark that, independently, also Honold, Kiermaier and Kurz proved that $\cA_q(5,3) = 2(q^3+1)$. We refer to their article \cite{TH:17} for their construction method, and for many other results on subspace codes.

\subsection{Preliminaries}

A {\em partial line spread} of $\PG(4,q)$ is a set of pairwise disjoint lines of $\PG(4,q)$. From $\cite{B}$, the largest partial line spread of $\PG(4,q)$ has size $q^3+1$. Let $\cC$ be an optimal $(5,3)_q$ subspace code. Since the lines contained in $\cC$ are pairwise disjoint, it follows that $\cC$ contains at most $q^3+1$ lines. A dual argument shows that $\cC$ contains at most $q^3+1$ planes. Hence, if $\cC$ consists of lines and planes, we have that $| \cC | \le 2(q^3+1)$  and, in order to construct such a code one needs to find a set $\cL$ of $q^3+1$ pairwise skew lines and a set $\cP$ of $q^3+1$ planes mutually intersecting in exactly a point such that no line of $\cL$ is contained in a plane of $\cP$.   

Notice that $\cC$ contains at most one point and, dually, $\cC$ contains at most one solid. The next result gives an upper bound on the size of $\cC$. 

\begin{lemma}
If $\cC$ contains a point, then $\cC$ contains at most $q^3$ planes. Dually, if $\cC$ contains a solid, then $\cC$ contains at most $q^3$ lines. 
\end{lemma}
\begin{proof}
We only need to prove the second assertion. Assume that $\cC$ contains more than $q^3$ lines. From the discussion above, $\cC$ contains a set $L$ consisting of $q^3+1$ lines. Let $L'$ be the set of points of $\PG(4,q)$ covered by the lines of $L$. Then every solid is covered by at least a member of $L$. Indeed, if $S$ is a solid of $\PG(4,q)$ and $s$ is the number of lines of $L$ contained in $S$, we have that $|S \setminus L'| = q^3+q^2+q+1 - s(q+1)-(q^3+1-s) \le q^2 = |\PG(4,q) \setminus L'|$. Hence, $s \ge 1$. It follows that $\cC$ contains no solids, as required.
\end{proof}

From the previous result, it follows that $\cA_q(5,3) \le 2(q^3+1)$ and there are four possibilities for the code $\cC$:
\begin{itemize}
\item[I)] $\cC$ consists of one point, $q^3+1$ lines and $q^3$ planes;
\item[II)] $\cC$ consists of $q^3$ lines, $q^3+1$ planes and one solid;
\item[III)] $\cC$ consists of one point, $q^3$ lines, $q^3$ planes and one solid;
\item[IV)] $\cC$ consists of $q^3+1$ lines and $q^3+1$ planes.
\end{itemize} 
In the remaining part of the paper, we exhibit examples of $(5,2(q^3+1),3)_q$ codes showing that $\cA_q(5,3) = 2(q^3+1)$. These results are achieved by using a suitable subgroup $G$ of $\PGL(5,q)$ of order $q^3$, if $q$ is odd, or of order $q^3-q$, if $q$ is even. We shall find it helpful to work with the elements of $G$ as matrices in $\GL(5,q)$. We shall consider the points of $\PG(4,q)$ as column vectors, with matrices in $G$ acting on the left.

\section{The odd characteristic case}

Let $q = p^h$, where $p$ is an odd prime. Let $\PG(4,q)$ be the four dimensional projective space over $\GF(q)$ equipped with homogeneous coordinates $(X_1, X_2, X_3, X_4, X_5)$, let $\pi$ be the projective plane with equations $X_4 = X_5 = 0$ and let $\ell$ be the line of $\pi$ with equations $X_3 = X_4 = X_5 = 0$. Let $\omega$ be a primitive element of $\GF(q)$ and denote by $\Pi_i$ the solid of $\PG(4,q)$ passing through $\pi$ with equation $X_4 = \omega^{i-1} X_5$, if $1 \le i \le q-1$, or $X_4 = 0$, if $i = q$, or $X_5 = 0$, if $i = q+1$. 

\begin{lemma}
Let $a,b,c$ be fixed elements of $\GF(q)$ such that the polynomial $X^3+aX^2+bX+c = 0$ is irreducible over $\GF(q)$ and 
$$
M_{r,s,t} = \left(
\begin{array}{ccccc}
1 & 0 & r & r^2-ar+s & t \\
0 & 1 & s & 2rs-t & s^2+bs-cr \\
0 & 0 & 1 & 2r & 2s \\
0 & 0 & 0 & 1 & 0 \\
0 & 0 & 0 & 0 & 1
\end{array}
\right) ,
$$
then $G = \{ M_{r,s,t} \;\; | \;\;  r,s,t \in \GF(q) \}$ is a $p$--group of order $q^3$.
\end{lemma}
\begin{proof}
Notice that $M_{r,s,t} M_{r',s',t'}^{-1} = M_{r,s,t} M_{-r',-s',2r's'-t'} = M_{r-r',s-s',t-t'-2s'(r-r')}$. On the other hand, since $M_{r,s,t}^n = M_{nr,ns,nt+n(n-1)rs}$, we have that each element of $G$ distinct from the identity has order $p$.    
\end{proof}

\begin{remark}
The group $G$ is not abelian, in particular the map $M_{r,s,t} \mapsto \left( \begin{array}{ccc} 1 & r & t/2 \\ 0 & 1 & s \\ 0 & 0 & 1 \end{array} \right)$ gives an isomorphism between $G$ and the so called Heisenberg group.
\end{remark}

We need the following preliminary results.

\begin{lemma}\label{points}
The group $G$ has $2q+3$ orbits on the points of $\PG(4,q)$:
\begin{itemize}
\item[a)] $q+1$ orbits of size one, each consisting of a point of $\ell$,
\item[b)] one orbit of size $q^2$, consisting of the points of $\pi \setminus \ell$,
\item[c)] $q+1$ orbits of size $q^3$, each consisting of the points of the set $\Pi_i \setminus \pi$. 
\end{itemize}
\end{lemma}
\begin{proof}
To describe the orbits on the points of $\pi$ is trivial. Since the group $G$ fixes every solid through the plane $\pi$, we need only to prove that the stabilizer in $G$ of a point of $P \in \Pi_i \setminus \pi$ is trivial. If $P = (0,0,0,\omega^{i-1},1)$, then $P^{M_{r,s,t}} = (\omega^{i-1}(r^2-ar+s)+t,\omega^{i-1}(2rs-t)+s^2+bs-cr,2r\omega^{i-1}+2s,\omega^{i-1},1)$. It follows that $M_{r,s,t}$ fixes $P$ if and only if
$$
\left\{ 
\begin{array}{l}
r\omega^{i-1}+s = 0 \\
\omega^{i-1}(r^2-ar+s)+t = 0 \\
\omega^{i-1}(2rs-t)+s^2+bs-cr = 0.
\end{array}
\right.
$$
Calculating $s$ and $t$ from the first and second equation, respectively, and substituting in the third equation we obtain $r({\omega^{3(i-1)}}+a {\omega^{2(i-1)}}+b \omega^{i-1}+c) = 0$. Hence it forces $M_{r,s,t}$ to be the identity. Analogously, if $P = (0,0,0,1,0)$ or $P = (0,0,0,0,1)$.  
\end{proof}

\begin{lemma}\label{lines}
The group $G$ has the following orbits on the lines of $\PG(4,q)$ distinct from $\ell$:
\begin{itemize}
\item[a)] $q+1$ orbits of size $q$, each consisting of the lines of $\pi$ passing through a point of $\ell$,
\item[b)] $(q+1)^2$ orbits of size $q^2$, each consisting of lines of $\Pi_i$, not in $\pi$, passing through a point of $\ell$, $1 \le i \le q+1$, 
\item[c)] $q(q+1)$ orbits of size $q^3$, each consisting of lines of $\Pi_i$ skew to $\ell$, $1 \le i \le q+1$, but intersecting $\pi$ in a point,
\item[d)] $q^3$ orbits of size $q^3$, each consisting of lines that are disjoint from $\pi$. 
\end{itemize}
In particular, every line--orbit of type $d)$ is a partial line spread.
\end{lemma}
\begin{proof}
To describe orbits of type $a)$ or $b)$ is trivial. 

\medskip
\fbox{Line--orbits of type $c)$}
\medskip

\par\noindent
Let $x$ be a line of $\Pi_i$, for some $i$, such that $x$ is skew to $\ell$ and let $X = x \cap \pi$. The stabilizer of $X$ in $G$, say $G_X$, is the group of order $q$ consisting of elements of type $M_{0,0,t}$. In particular, $G_X$ is an elation group having as axis the plane $\pi$ and as center the line $\ell$, and every orbit of $G_X$ on points not in $\pi$ consists of $q$ points of a line meeting $\ell$ in one of its points. It follows that $x^{G_X}$ consists of $q$ lines through $X$ in a plane. Since the group $G$ permutes the $q^2$ points of $\pi \setminus \ell$ in a single orbit, we have that the orbit $x^G$ contains $q^3$ lines.

\medskip
\fbox{Line--orbits of type $d)$}
\medskip

\par\noindent
Let $y$ be a line skew to $\pi$. Since $y$ cannot be contained in a hyperplane $\Pi_i$, $1 \le i \le q+1$, we have that $y$ meets $\Pi_i$ in a point, say $Y_i$, $1 \le i \le q+1$. The stabilizer of $y$ in $G$ is trivial. Indeed, assume on the contrary that $id \ne g \in G$ fixes the line $y$. Then, since $G$ fixes $\Pi_i$, we would have that $g$ fixes $y \cap \Pi_i = Y_i$, which is a contradiction, since, from Lemma \ref{points}, $G$ acts sharply transitive on the points of $\Pi_i \setminus \pi$. It follows that the orbit $y^G$ consists of $q^3$ lines such that the points covered by its members are all the $q^4+q^3$ points of $\PG(4,q) \setminus \pi$. As a consequence, lines in $y^G$ are pairwise skew.    
\end{proof}

We are ready to prove our main result of this section.

\begin{theorem}\label{code}
There exists a set $\cL$ consisting of $q^3+1$ pairwise skew lines and a set $\cP$ consisting of $q^3+1$ planes mutually intersecting in exactly a point, such that no line of $\cL$ is contained in a plane of $\cP$. 
\end{theorem}
\begin{proof}
Let $\alpha$ be a plane of $\PG(4,q)$ such that $\alpha \cap \pi$ is a point $A$ not belonging to $\ell$. Then $\alpha \cap \Pi_i$, $1 \le i \le q+1$, is a line passing through the point $A$. Let $\cP' = \alpha^G$. We claim that $\cP'$ is a set consisting of $q^3$ planes pairwise intersecting in a point. Indeed, assume on the contrary that there exist two distinct planes $\alpha, \alpha'$ in $\cP'$ such that $\alpha \cap \alpha'$ is a line, say $b$. Let $B$ be the line $G$--orbit containing $b$. Two possibilities occur: either $b$ meets $\pi$ in a point not on $\ell$, or $b$ is disjoint from $\pi$. 

If the former case occurs, there exists a solid $\Pi_i$ such that $b \in \Pi_i$, for some $i$, and $B$ is a line--orbit of type $c)$. Consider the tactical configuration whose points are the elements of $B$ and whose blocks are the elements of $\cP'$. Since $|B| = |\cP'| = q^3$ and through $b$, there pass at least two distinct elements of $\cP'$, it follows that $\alpha$ contains at least two distinct elements of $B$. Hence $\alpha$ is a plane of $\Pi_i$ and so $\alpha \cap \pi$ is a line, a contradiction. 

If the latter case occurs, then $B$ is a line--orbit of type $d)$. Consider again the tactical configuration whose points are the elements of $B$ and whose blocks are the elements of $\cP'$. Arguing as in the previous case, $\alpha$ contains at least two distinct elements of $B$, say $b_1$ and $b_2$, but then $b_1$ meets $b_2$ in a point, contradicting the fact that two distinct lines in $B$ are skew. The plane $\alpha$ contains $q^2$ lines that are disjoint from $\pi$ and each one of them belongs to a distinct line--orbit of type $d)$. Since there are $q^3$ line--orbits of type $d)$, it follows that we can find a line--orbit of type $d)$, say $\cL'$, such that no line in $\cL'$ is contained in $\alpha$. 

Let $r$ be a line of $\pi$ distinct from $\ell$ and let $\cL:=\cL' \cup \{ r \}$. Then $\cL$ is a set consisting of $q^3+1$ pairwise skew lines. Let $\xi$ be a plane passing through the line $\ell$ distinct from $\pi$ and let $\cP := \cP' \cup \{ \xi \}$. Then, every plane in $\cP'$ shares a point with $\xi$. Indeed, if there was a plane $\beta \in \cP'$ meeting $\xi$ in a line, then $\xi \cap \pi \in \ell$, a contradiction. Hence, $\cP$ is a set consisting of $q^3+1$ planes mutually intersecting in a point. In particular, no line of $\cL$ is contained in a plane of $\cP$, as required.   
\end{proof}

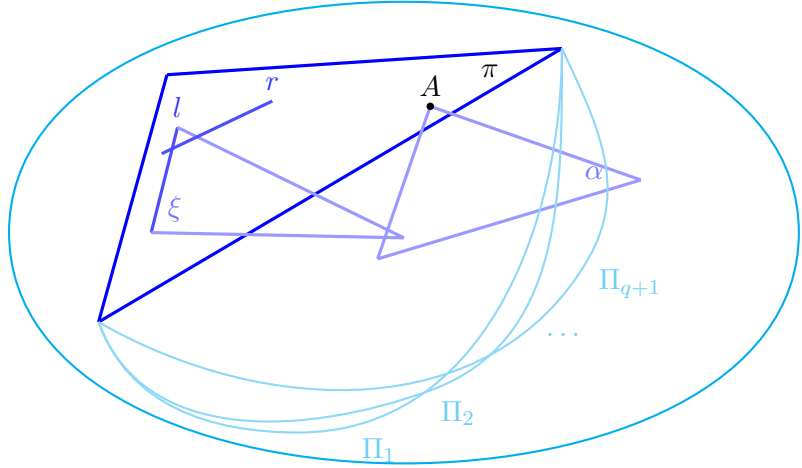
\begin{figure}[h!]
\centering
\begin{tikzpicture}[scale=0.7]
%\draw [very thick, blue] (0,0) coordinate (a1) --(3.8,4) coordinate (b1);
\draw [thick, cyan] (0,0)  to [out=90,in=90] (15,0) to [out=270,in=270] (0,0); 
\draw [very thick, blue] (3,3)  --(10.5,3.5) coordinate (b) ;
\draw [very thick, blue] (3,3)  --(1.7,-1.7) coordinate (a) ;
\draw [very thick, blue] (1.7,-1.7)  --(10.5,3.5) ;
\node [below left, thick] at (9.5,3.4) {$\pi$};

\draw [very thick, blue!40] (8,2.4)  --(12,1) ;
\draw [very thick, blue!40] (12,1)  --(7,-0.5) ;
\draw [very thick, blue!40] (7,-0.5)  --(8,2.4) ;
\node [above, thick] at (8,2.4) {$A$};
\node [above left, very thick, blue!55] at (11.5,0.8) {$\alpha$};
\node[fill=black,circle,inner sep=1pt] at (8,2.4) {};

\draw [very thick, blue!70] (3.2,2)  --(2.7,0) ;
\draw [very thick, blue!40] (2.7,0)  --(7.5,-0.1) ;
\draw [very thick, blue!40] (7.5,-0.1)  --(3.2,2) ;
\node [above, thick, blue!80] at (3.2,2) {$l$};
\node [above right, very thick, blue!55] at (2.8,0) {$\xi$};

\draw [very thick, blue!70] (2.9,1.5)  --(5,2.5) ;
\node [above, thick, blue!80] at (5,2.5) {$r$};

\draw [thick, cyan!40] (a)  to [out=290,in=180] (5.5,-3.8) to [out=0,in=270] (b); 
\draw [thick, cyan!40] (a)  to [out=290,in=200] (8,-3) to [out=20,in=270] (b); 
%\draw [thick, cyan!60] (a)  to [out=330,in=250] (12.5,-0.7) to [out=70,in=295] (b); 
\draw [thick, cyan!40] (a)  to [out=330,in=240] (11,-0.5) to [out=60,in=295] (b);
\node [below right, thick, cyan!50] at (6.5,-3.7)  {$\Pi_1$};
\node [below right, thick, cyan!50] at (8,-3) {$\Pi_2$};
\node [below right, thick, cyan!50] at (11,-0.5)  {$\Pi_{q+1}$};
\node [below right, thick, cyan!50] at (10,-1.7)  {$\ldots$};
\end{tikzpicture}
\caption{Construction for $q$ odd} \label{qoneven}
\end{figure}

Figure \ref{qoneven} illustrates the different subspaces used in this code construction in PG$(4,q)$, $q$ odd. The figure shows the plane $\pi$ and the line $\ell$, stabilized by the automorphism group $G$. It shows the hyperplanes $\Pi_i$ passing through the plane $\pi$. The figure also shows one plane $\alpha$ of the orbit of planes in the code. Finally, the figure also shows the line $r$ and the plane $\xi$ through the line $\ell$, which are added to the code as the final two codewords.

\begin{remark}
With the same notation as in Theorem \ref{code}, the code $\cC = \cP \cup \cL$ is an optimal code of type $IV)$. Let $X$ be a point of $\ell \setminus r$, then the code $\cC = \cP' \cup \cL \cup \{ X \}$ is an optimal code of type $I)$. Let $j$ be such that $\Pi_j = \langle \pi, \xi \rangle$, then the code $\cC = \cP \cup \cL' \cup \{ \Pi_k \}$, $k \ne j$, is an optimal code of type $II)$. Finally, the code $\cC = \cP' \cup \cL' \cup \{ X, \Pi_i \}$ is an optimal code of type $III)$.
\end{remark}

\section{The even characteristic case}

Let $q = 2^h$. Let $\PG(4,q)$ be the four dimensional projective space over $\GF(q)$ equipped with homogeneous coordinates $(X_1, X_2, X_3, X_4, X_5)$, let $\pi$ be the projective plane with equations $X_4 = X_5 = 0$ and let $\ell$ be the line of $\pi$ with equations $X_1 = X_4 = X_5 = 0$. Denote by $\Sigma$ the solid of $\PG(4,q)$ with equation $X_1 = 0$. Then $\Sigma \cap \pi = \ell$. Let $\cH$  be  the hyperbolic quadric of $\Sigma$ having equation $X_2 X_5 + X_3 X_4 = 0$. 

Let $\alpha \in \GF(q)$ be such that the polynomial $X^2+X+\alpha = 0$ is irreducible over $\GF(q)$ and let $G = \{ M_{a,b,c,d} \;\; | \;\;  a \in \GF(q) \setminus \{ 0 \}, b,c,d \in \GF(q), c^2+cd+\alpha d^2 = 1 \}$, where
$$
M_{a,b,c,d} = \left(
\begin{array}{ccccc}
1 & 0 & 0 & 0 & 0 \\
0 & ac & \alpha ad & bc & \alpha bd \\
0 & ad & a(c+d) & bd & b(c+d) \\
0 & 0 & 0 & a^{-1}c & \alpha a^{-1}d \\
0 & 0 & 0 & a^{-1}d & a^{-1}(c+d)
\end{array}
\right) .
$$
Then $G$ is a group of order $q^3-q$ with structure $C_{q+1} \times (E_q \times C_{q-1})$. In particular the subgroup $G_1 = \{M_{1,0,c,d} \;\; | \;\; c,d \in \GF(q), c^2+cd+\alpha d^2 =1\}$ is a cyclic group of order $q+1$, the subgroup $G_2 = \{M_{1,b,1,0} \;\; | \;\; b \in \GF(q)\}$ is an elementary abelian group of order $q$ and the subgroup $G_3 = \{M_{a,0,1,0} \;\; | \;\; a \in \GF(q) \setminus \{0\}\}$ is a cyclic group of order $q-1$.

As we will see in the next lemma, the group $G$ fixes a pencil of quadrics, say $\cF$. The pencil $\cF$ comprises $q-1$ parabolic quadrics $\cQ_i$, $1 \le i \le q-1$, the solid $\Sigma$ and the cone $\cC$ having as vertex the point $N = (1,0,0,0,0)$ and as base $\cH$. The base locus of $\cF$ is $\cH$ and $\cF$ is generated by the quadrics having equation $X_1^2 = 0$ and $X_2 X_5 + X_3 X_4 = 0$. The nucleus of the parabolic quadric $\cQ_i$ is the point $N$, $1 \le i \le q-1$. 

We need the following preliminary results.

\begin{lemma}\label{point}
The group $G$ has $q+5$ orbits on points of $\PG(4,q)$:
\begin{itemize}
\item[a)] the point $N$,
\item[b)] an orbit of size $q+1$, consisting of the points of $\ell$,
\item[c)] an orbit of size $q^2-1$, consisting of the points of $\pi \setminus (\ell \cup \{ N \})$,
\item[d)] an orbit of size $q^2+q$, consisting of the points of $\cH \setminus \ell$,
\item[e)] an orbit of size $q^3-q$, consisting of the points of $\Sigma \setminus \cH$,
\item[f)] an orbit of size $q^3-q$, consisting of the points of $\cC \setminus (\pi \cup \cH)$,
\item[g)] $q-1$ orbits of size $q^3-q$, each consisting of the points of $\cQ_i \setminus \cH$, $1 \le i \le q-1$.   
\end{itemize}
\end{lemma}
\begin{proof}
The group $G$ fixes the point $N$ and the line $\ell$. On the other hand, the subgroup $G_1$ acts transitively on the points of $\ell$. If $P \in \pi \setminus (\ell \cup \{ N \})$, then the stabilizer of $P$ in $G$ is the subgroup $G_2$. Whereas, if $P = (0,0,0,0,1,0) \in \cH \setminus \ell$, then the stabilizer of $P$ in $G$ is the subgroup $G_3$. Finally, straightforward computations show that if $g \in G$ and $P$ belongs either to $\Sigma \setminus \cH$, or to $\cC \setminus (\pi \cup \cH)$, or to $\cQ_i \setminus \cH$, $1 \le i \le q-1$, then $P^g$ is a point of $\Sigma \setminus \cH$, $\cC \setminus (\pi \cup \cH)$, $\cQ_i \setminus \cH$, $1 \le i \le q-1$, respectively and $P^g = P$ if and only if $g$ is the identity. 
\end{proof}

Let $\cR_1$ be the regulus of $\cH$ containing $\ell$ and let $\cR_2$ be its opposite regulus. Since $q$ is even, the lines of $\Sigma$ that are tangent to $\cH$, together with the $2(q+1)$ lines of $\cR_1 \cup \cR_2$, are the $(q+1)(q^2+1)$ lines of a symplectic polar space $\cW(3,q)$ of $\Sigma$. See \cite{H1} for more details. Since the group $G$ fixes $\cH$, it follows that $G$ fixes $\cW(3,q)$, also. We denote by $\cT$ the set of $q^3-q$ lines of $\cW(3,q)$ having exactly one point in common with $\cH \setminus \ell$. In the next results we summarize useful information regarding the action of the group $G$ and its subgroups.

\begin{lemma}\label{reg}
Let $t$ be a line of $\cT$, then $t^{G_2} \cup \{ \ell \}$ is a regulus of $\Sigma$.
\end{lemma}
\begin{proof}
The group $G_2$ fixes pointwise the plane $\pi:X_4=X_5=0$ and in particular the line $\ell$. Since $G_2$ fixes $\cW(3,q)$ it follows that the parabolic congruence $\rho$ of $\cW(3,q)$ having as axis the line $\ell$ is fixed by $G_2$. A straightforward computation shows that every line of $\rho$, distinct from $\ell$, is fixed by $G_2$. In particular, if $r$ is  a line of $\rho \setminus \{\ell\}$, then $G_2$ permutes the $q$ points of $r \setminus \{r \cap \ell\}$ in a single orbit. Let $t \in \cT$. Then $t$ is disjoint from $\ell$. Since through every point of $\Sigma \setminus \{\ell\}$ there is exactly one line of the parabolic congruence $\rho$ containing it, we have that there are $q+1$ lines of $\rho$, say $r_1, \dots, r_{q+1}$, each meeting $t$ in a distinct point. The lines $r_1, \dots, r_{q+1}$ are pairwise disjoint, otherwise there would exist a triangle consisting of lines of $\cW(3,q)$, contradicting the fact that $\cW(3,q)$ is a generalized quadrangle. The assertion, now, follows from the fact that $t^{G_2} \cup \{ \ell \}$ consists of $q+1$ lines covering $(q+1)^2$ points of $\Sigma$ and meeting three pairwise disjoint lines.   
\end{proof}

\begin{lemma}\label{tan}
The group $G$ acts transitively on the elements of $\cT$. 
\end{lemma}
\begin{proof}
Since $G$ acts transitively on the points of $\cH \setminus \{ \ell \}$, we only need to prove that $G$ permutes the $q-1$ tangent lines that are concurrent at some point of $\cH \setminus \ell$. Let $U_4 = (0,0,0,1,0) \in \cH \setminus \ell$. The lines of $\cH$ containing $U_4$ are $r_1 = U_4 U_5 \in \cR_1$ and $r_2 = U_2 U_4 \in \cR_2$. The lines of $\cT$ containing the point $U_4$ lie in the plane $\langle r_1, r_2 \rangle$. The stabilizer of $U_4$ in $G$, i.e. $G_3$, fixes the plane $\langle r_1, r_2 \rangle$ and induces a homology group of order $q-1$ having as axis the line $r_1$ and as center the point $U_2$. Hence, the $q-1$ lines of $\cT$ passing through $U_4$ are permuted in a single orbit under the action of $G_3$.     
\end{proof}

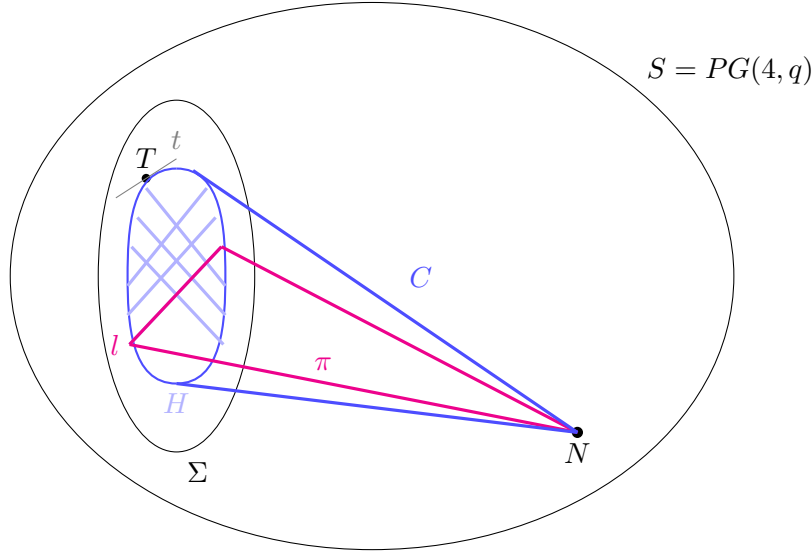
\begin{figure}[h!] 
\centering
\begin{tikzpicture}[scale = 1.3]
%\draw[help lines, cyan!50] (0,0) grid (5,5);

\node[fill=black, circle,inner sep=1.5pt] at (4.1,-0.5){};
\node [below, thick] at (4.1,-0.5) {$N$};
\node [above, gray] at (0,2.3) {$t$};
\node [above] at (-0.31,2.1) {$T$};
\node [left, magenta] at (-0.475,0.4) {$l$};
\node[fill=black, circle,inner sep=1.2pt] at (-0.31,2.1){};
\draw (0,1.1) ellipse (0.8 and 1.8);
\draw (2,1.1) ellipse (3.7 and 2.8);

\node [below right, thick] at (0,-0.7) {$\Sigma$};
\node [below, thick, blue!70] at (2.5,1.3) {$C$};
\node [ thick, magenta] at (1.5,0.2) {$\pi$};
\node [above, right, thick] at (4.7,3.2) {$S = PG(4,q)$};
\draw [thick, blue!70] (0,0)  to [out=0,in=-90] (0.5,1.1) to [out=90,in=0] (0,2.2) ; 
\draw [thick, blue!70] (0,0)  to [out=180,in=-90] (-0.5,1.1) to [out=90,in=180] (0,2.2) ; 
\draw [gray](-0.62,1.9)--(-0.31,2.1) -- (0,2.3);
\draw [very thick, blue!30] (0.31,2)   --(-0.5,1);
\draw [very thick, blue!30] (0.4,1.7)   --(-0.5,0.7);
\draw [very thick, blue!30] (0.46,1.4)  coordinate (a)  --(-0.48,0.4) coordinate (b);
\draw [very thick, blue!30] (-0.31,2)  --(0.5,1) ;
\draw [very thick, blue!30] (-0.4,1.7)   --(0.5,0.7);
\draw [very thick, blue!30] (-0.46,1.4)   --(0.48,0.4);
\node [below,  blue!30] at (0,0) {$H$};

\draw [very thick, magenta] (a)   --(b);
\draw [very thick, magenta] (a)   --(4.1,-0.5);
\draw [very thick, magenta] (4.1,-0.5)   --(b);

\draw [very thick, blue!70] (0,0)   --(4.1,-0.5);
\draw [very thick, blue!70] (4.1,-0.5)   --(0.17,2.18);

\end{tikzpicture}
\caption{Construction for $q$ even}\label{qevengroot25}
\end{figure}

Figure \ref{qevengroot25} illustrates the initial elements required in the code construction in PG$(4,q)$, $q$ even. It shows the plane $\pi$, the line $\ell$, and the  nucleus $N$, lying in the plane $\pi$, of all the quadrics $\mathcal{Q}_i$,  $1\leq i\leq q-1$.  It shows the hyperbolic quadric ${\mathcal H}$ contained in the solid $\Sigma$. It shows the cone $\mathcal{C}$ with base ${\mathcal H}$ and with vertex the point $N$. The figure also includes the tangent line $t$ to the hyperbolic quadric ${\mathcal H}$, sharing the tangent point $T$ with ${\mathcal H}$.

\begin{prop}\label{line}
Under the action of $G$, there are $q^2-q$ line--orbits of size $q^3-q$ consisting of mutually disjoint lines. 
\end{prop}
\begin{proof}
First of all notice that if $r$ is a line meeting each of the sets $\Sigma \setminus \cH, \cC \setminus (\pi \cup \cH), \cQ_i \setminus \cH$, $1 \le i \le q-1$, in exactly one point, then $r^G$ is a set consisting of $q^3-q$ lines covering $(q+1)(q^3-q)$ points and hence forming a partial spread of the ambient projective space. In the following we prove that there is a plane containing $q^2-q$ lines having the desired property. Since two lines in a projective plane have a point in common, no two of them are in the same  $G$--orbit. Hence, they are representatives of $q^2-q$ distinct line--orbits of size $q^3-q$ forming a partial spread. Let $t$ be a line of $\cT$, where $T = t \cap \cH$, and consider the plane $\gamma = \langle N, t \rangle$. It follows that $\gamma$ meets $\Sigma$ and $\cC$ in the line $t$ and $N T$, respectively. On the other hand, a solid $\Gamma$ containing the plane $\gamma$ has to meet the parabolic quadric $\cQ_i$ in a quadratic cone, since $N \in \Gamma$ and $N$ is the nucleus of $\cQ_i$. Notice that every line passing through $N$ meets $\cQ_i$ in a point. Hence, it turns out that two possibilities occur: either $\gamma \cap \cQ_i$ is a conic, whose nucleus is $N$, or $\gamma \cap \cQ_i$ is a line passing through $T$. Since $t$ and $N T$ are two lines that are tangent to $\cQ_i$ at the point $T \in \cQ_i$, the former case cannot occur. Hence, $\gamma$ meets $\cQ_i$ in a line passing through $T$, $1 \le i \le q-1$. It follows that every line of $\gamma$ containing neither $T$ nor $N$ is a line meeting each of the sets $\Sigma \setminus \cH, \cC \setminus (\pi \cup \cH), \cQ_i$, $1 \le i \le q-1$, in exactly one point.  
\end{proof}

In order to prove the next result, we embed $S = \PG(4,q)$ as a hyperplane section of $\PG(5,q)$, the five--dimensional projective space over $\GF(q)$ equipped with homogeneous coordinates $(X_1, X_2, X_3, X_4, X_5, X_6)$. Let $\cK$ be the Klein quadric of $\PG(5,q)$ defined by the equation $X_1 X_6 + X_2 X_5+ X_3 X_4 = 0$. Notice that $\cK \cap S = \cC$. Also, the group $G$ is extended canonically to a subgroup $\bar G$ of the stabilizer of $\cK$ in $\PGL(6,q)$. In particular, ${\bar G} = \{ {\bar M}_{a,b,c,d} \;\; | \;\;  a \in \GF(q) \setminus \{ 0 \}, b,c,d \in \GF(q), c^2+cd+\alpha d^2 = 1 \}$, where
$$
{\bar M}_{a,b,c,d} = \left(
\begin{array}{cccccc}
1 & 0 & 0 & 0 & 0 & 0\\
0 & ac & \alpha ad & bc & \alpha bd & 0 \\
0 & ad & a(c+d) & bd & b(c+d) & 0 \\
0 & 0 & 0 & a^{-1}c & \alpha a^{-1}d & 0 \\
0 & 0 & 0 & a^{-1}d & a^{-1}(c+d) & 0 \\
0 & 0 & 0 & 0 & 0 & 1
\end{array}
\right) .
$$
We will use the polarity $\perp$ of $\PG(5,q)$ associated to the Klein quadric $\cK$, to show that the existence of line--orbits consisting of mutually disjoint lines implies the existence of plane--orbits consisting of planes mutually intersecting in one point. 

\begin{prop}\label{plane}
Under the action of $G$, there are $q^2-q$ plane--orbits of size $q^3-q$ consisting of planes mutually intersecting in one point. 
\end{prop}
\begin{proof}
Let $S'$ be the hyperplane of $\PG(5,q)$ with equation $X_1 = 0$. Then, $S \cap S' = \Sigma$ and the Klein quadric $\cK$ meets $S'$ in a cone, say $\cC'$, having as vertex the point $N' = (0,0,0,0,0,1)$ and as base $\cH$. The group $\bar G$ acts identically on the line $N N'$, hence every hyperplane through $\Sigma = (N N')^\perp$ is fixed by $\bar G$. Also, it is clear that the action of $\bar G$ on $S'$ is the same as the action of $G$ on $S$. Therefore, by Proposition \ref{line}, under the action of $\bar G$, there are $q^2-q$ line--orbits of size $q^3-q$ consisting of mutually disjoint lines of $S'$. Let $\cR$ be one of these $q^2-q$ orbits. Then, each of the lines in $\cR$ is disjoint from $\pi' = \langle N', \ell \rangle$, meets $\Sigma$ in a point that is not in $\cH$, and has only one point in common with $\cC'$. Consider the plane $\sigma = \langle N, r \rangle$, where $r$ is a line of $\cR$. Consider the set $\cS = \{x^\perp \;\; | \;\; x \in \sigma^{\bar G}\}$. Since every plane in $\sigma^{\bar G}$ contains the point $N$, it follows that every plane in $\cS$ lies in $S$. On the other hand, since two distinct planes in $\sigma^{\bar G}$ have exactly the point $N$ in common, it follows that two distinct planes in $\cS$ generate the hyperplane $S$ and hence they meet exactly in a point. 
\end{proof}

With the same notation as in Proposition \ref{plane}, we want to provide an alternative description of the $q^2-q$ plane--orbits of size $q^3-q$ consisting of planes mutually intersecting in one point. Firstly, we prove the following lemma.

	\begin{lemma}\label{conic}
$\sigma \cap \cK$ is a conic
	\end{lemma}
	\begin{proof}
Notice that $N^\perp = S$, $N'^\perp = S'$, $\sigma \cap S$ is a line tangent to $\cK$ at the point $N$ and $\sigma$ contains at least two points of $\cK$, namely the points $N$ and $r \cap \cC'$. It follows that $\sigma \cap \cK$ is either a line, or two lines or a conic. If one of the first two cases occurs, then $\sigma$ would contain a line of $\cK$ through $N$, but all the lines of $\cK$ through $N$ lie in $S$, a contradiction. Hence, $\sigma \cap \cK$ is a conic.
	\end{proof}

\begin{remark}\label{description} 
From Lemma \ref{conic}, it follows that every plane in $\cS$ meets $\cC$ in a conic. Indeed, $\sigma \cap \cK$ is a conic and $\bar G$ stabilizes $\cK$. Consider the solid $\Sigma' = \langle \sigma, N' \rangle$. Since $\langle \Sigma, N, N' \rangle$ is the whole ambient projective space $\PG(5,q)$, and $\Sigma'$ contains both $N$, $N'$, we have that $\Sigma \cap \Sigma'$ is a line. From the proof of Proposition \ref{line}, if $r$ is the line of $\cR$ such that $\sigma = \langle N, r \rangle$, then the plane $\langle N', r \rangle$ meets $\Sigma$ in a line $t$ that is tangent to $\cH$ at the point $U$. Hence, $\Sigma \cap \Sigma' = t$. Also, since $\Sigma^\perp \subset \Sigma'$, we have that $\Sigma'^\perp \subset \Sigma$ and, hence, $\Sigma'^\perp = \Sigma'^\perp\cap \Sigma$. Moreover, $\Sigma' \cap \Sigma = (\Sigma' \cap \Sigma)^{\perp_{|\Sigma}} = t^{\perp_{|\Sigma}} = t$. Notice that, since $q$ is even, $\perp_{|\Sigma}$ coincides with the polarity of $\Sigma$ associated with $\cW(3,q)$. We have that $\sigma^\perp \cap \Sigma = (\sigma \cup \Sigma^\perp)^\perp = \Sigma'^\perp = t$ and every plane in $\cS$ meets $\Sigma$ in a line of $\cT$.       
Let $U = t \cap (\cH \setminus \ell)$. Let $R$ be the point in common between the unique line in $\cR_2$ containing $U$ and $\ell$. Let $U' = \sigma^\perp \cap \pi$. Since $\sigma^\perp \cap \Sigma = t$, it follows that $U' \notin \ell$. On the other hand, if $U' \in R N$, then the line $U U'$ would be contained in $\sigma^\perp \cap \cC$, contradicting the fact that $\sigma^\perp \cap \cC$ is a conic. 
In other words, let $t$ be a line of $\cT$ and let $U = t \cap (\cH \setminus \ell)$. Let $R$ be the point in common between the unique line in $\cR_2$ containing $U$ and $\ell$. Let $U'$ be a point of $\pi \setminus (\ell \cup R N)$. Let $\gamma$ be the plane containing the line $t$ and the point $U'$. Then $\gamma^G$ consists of $q^3-q$ planes mutually intersecting in one point. There are $q^2-q$ choices for the point $U'$, and taking into account Lemma \ref{tan}, each of these points gives rise to a representative for a plane--orbit of size $q^3-q$ consisting of planes mutually intersecting in one point.
\end{remark}

\begin{figure}[h!] 
\centering
\begin{tikzpicture}

%\draw[help lines, cyan!50] (-5,-5) grid (5,5);
%\node[fill=magenta,circle,inner sep=1.5pt] at (0,0){};

\draw [thick, cyan] (-5.5,3.5)  to [out=290,in=70] (-5.5,-3.5) to [out=110,in=250] (-5.5,3.5) ; 
\draw [thick, cyan] (5.5,3.5)  to [out=290,in=70] (5.5,-3.5) to [out=110,in=250] (5.5,3.5) ; 

\draw[thick, cyan] (-5.5,-3.5) .. controls(0,-2) .. (5.5,-3.5) ;
\draw[thick, cyan] (-5.5,3.5) .. controls(0,2) .. (5.5,3.5) ;
\node [left, thick, cyan] at (-5.5,3.5) {$\mathcal{K}=Q^+(5,q)$};

\draw [rotate=-65.8](1,1) ellipse (4 and 1.5);
\draw [rotate=65.8](-1,1) ellipse (4 and 1.5);
\node [left, thick] at (3,-4.5) {$S$};
\node [right, thick] at (-3,-4.5) {$S'$};

\node[fill=magenta,circle,inner sep=1.5pt] at (3,-2.5){};
\node[fill=magenta,circle,inner sep=1.5pt] at (-3,-2.5){};
\node [left] at (-3,-2.5) {$N'$};
\node [right] at (3,-2.5) {$N$};

\draw [thick, blue] (0,0)  to [out=0,in=-90] (0.5,1.1) to [out=90,in=0] (0,2.2) ; 
\draw [thick, blue] (0,0)  to [out=180,in=-90] (-0.5,1.1) to [out=90,in=180] (0,2.2) ; 
\draw [very thick, blue!30] (0.31,2)   --(-0.5,1);
\draw [very thick, blue!30] (0.4,1.7)   --(-0.5,0.7);
\draw [very thick, blue!30] (0.46,1.4)   --(-0.48,0.4);
\draw [very thick, blue!30] (-0.31,2)   --(0.5,1);
\draw [very thick, blue!30] (-0.4,1.7)   --(0.5,0.7);
\draw [very thick, blue!30] (-0.46,1.4)   --(0.48,0.4);
\node [right, blue!40] at (0.1,2.2) {$H$};

\draw [very thick, blue!70] (-3,-2.5)  --(0,0) ;
\draw [very thick, blue!70] (3,-2.5)  --(0,0) ;
\draw [very thick, blue!70] (0.31,2)  --(3,-2.5) ;
\draw [very thick, blue!70] (-0.31,2)  --(-3,-2.5) ;
\node [ blue] at (1.9,0) {$\mathcal{C}$};
\node [ blue] at (-1.9,0) {$\mathcal{C'}$};

\draw [very thick, black!60!green] (-1.2,-0.5) coordinate (r1)  --(0,-0.5) coordinate (r2);
\node[fill=black!60!green,circle,inner sep=1.2pt] at (r1){};
\node[fill=black!60!green,circle,inner sep=1.2pt] at (r2){};
\draw [very thick, black!60!green] (3,-2.5)   --(r2);
\draw [very thick, black!60!green] (3,-2.5)   --(r1);
\node [black!60!green] at (0.45,-1.1) {$\sigma$};
\node [black!60!green, above left] at (-0.6,-0.5) {$r$};
\end{tikzpicture}
\caption{Construction in  PG$(5,q)$, $q$ even}\label{qevengroot}
\end{figure}
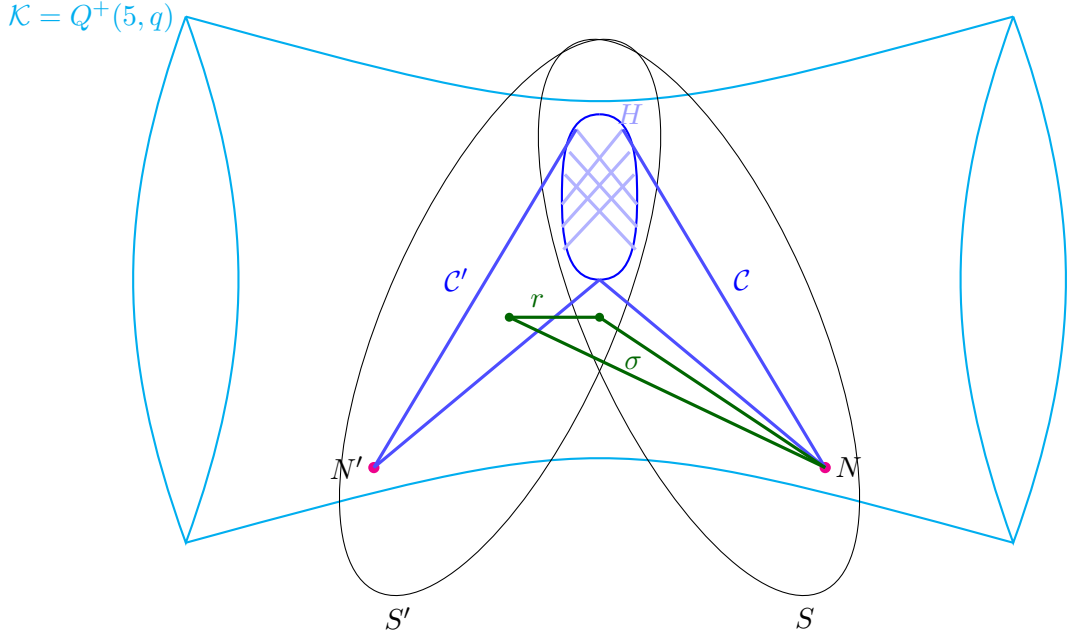

Figure \ref{qevengroot} aims at giving a global overview of the setting in PG$(5,q)$, $q$ even, regarding the construction of the subspace code. We see the Klein quadric ${\mathcal K}$, and the two hyperplanes $S$ and $S'$. We see the 3-dimensional hyperbolic quadric ${\mathcal H}$, contained in the intersection of $S$ and $S'$. We see the two cones ${\mathcal C}$ and ${\mathcal C'}$ with base ${\mathcal H}$ and with respective vertices $N$ and $N'$. Figure \ref{qevengroot}
  also illustrates the plane $\sigma =\langle r,N\rangle$. \\

With the same notation as in Remark \ref{description}, let $\gamma$ be the plane $\langle t, U' \rangle$ such that $\gamma^G$ consists of $q^3-q$ planes mutually intersecting in one point. 

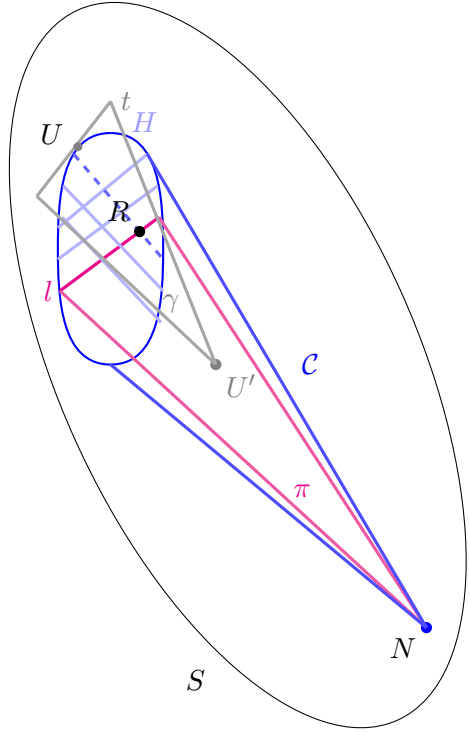
\begin{figure}[h!] 
\centering
\begin{tikzpicture}[scale = 1.4]

%\draw[help lines, cyan!50] (-3,-3) grid (5,5);

\draw [rotate=-65.8](0.5,1.1) ellipse (3.7 and 1.7);
\node [left, thick] at (1,-3) {$S$};

\node[fill=blue,circle,inner sep=1.5pt] at (3,-2.5){};
\node [below left] at (3,-2.5) {$N$};

\draw [thick, blue] (0,0)  to [out=0,in=-90] (0.5,1.1) to [out=90,in=0] (0,2.2) ; 
\draw [thick, blue] (0,0)  to [out=180,in=-90] (-0.5,1.1) to [out=90,in=180] (0,2.2) ; 
\draw [very thick, blue!30] (0.35,2)   --(-0.5,1.3);
\draw [very thick, blue!30] (0.45,1.7)   --(-0.5,1);
\draw [very thick, magenta] (0.46,1.4) coordinate (l1)   --(-0.48,0.7) coordinate (l2);
\node [left, magenta] at (-0.45,0.7) {$l$};
\node [left, magenta] at (2,-1.2) {$\pi$};
\draw [very thick, blue!60, style=dashed] (-0.35,2) coordinate (c1)   --(0.5,1) coordinate (c2);
\draw [very thick, blue!30] (-0.45,1.7)   --(0.5,0.7);
\draw [very thick, blue!30] (-0.46,1.4)   --(0.48,0.4);
\node [right, blue!40] at (0.1,2.3) {$H$};

\draw [very thick, magenta!80] (l1)   --(3,-2.5);
\draw [very thick, magenta!80] (l2)   --(3,-2.5);

\draw [very thick, blue!70] (3,-2.5)  --(0,0) ;
\draw [very thick, blue!70] (0.35,2)  --(3,-2.5) ;
\node [ blue] at (1.9,0) {$\mathcal{C}$};

\draw [very thick, gray!70] (-0.7,1.6) coordinate (t1)  --(0,2.5) coordinate (t2);
\node [ gray, right] at (t2) {$t$};
\node[fill=gray,circle,inner sep=1.2pt] at (-0.31,2.07){};
\node [ above left] at (-0.35,2) {$U$};
\node[fill=gray,circle,inner sep=1.5pt] at (1,0){};
\node [ gray, below right] at (1,0) {$U'$};
\draw [very thick, gray!70] (t1)  --(1,0) ;
\draw [very thick, gray!70] (t2)  --(1,0) ;
\node [ gray, above left] at (0.75,0.4) {$\gamma$};

\node[fill=black,circle,inner sep=1.5pt] at (intersection of l1--l2 and c1--c2) {};
\node[above left] at (intersection of l1--l2 and c1--c2)  {$R$};
\end{tikzpicture}

\caption{Construction in the hyperplane $S$ of PG$(5,q)$,  $q$ even} \label{qevenkleinx}
\end{figure}

Figure \ref{qevenkleinx} illustrates the construction in the hyperplane $S$. \\

We say that $\gamma^G$ is a {\em good} plane--orbit. Also, we say that a line--orbit of size $q^3-q$ consisting of mutually disjoint lines as described in Proposition \ref{line} is a {\em good} line--orbit.   

	\begin{prop}\label{pencil}
A plane in a good plane--orbit contains exactly $q-1$ lines of distinct good line--orbits.
	\end{prop}
	\begin{proof}
We need to show that there are exactly $q-1$ lines of $\gamma$ meeting each of the sets $\Sigma \setminus \cH, \cC \setminus (\pi \cup \cH), \cQ_i$, $1 \le i \le q-1$, in exactly one point. Notice that $\cF$ induces a pencil of plane quadrics $\cF'$ on $\gamma$ containing the line $t = \gamma \cap \Sigma$ and the conic $\gamma \cap \cC$, see Remark \ref{description}. Furthermore, the base locus of $\cF'$ is the point $U = \gamma \cap \cH$. Notice that $\gamma \cap \cQ_i$ cannot contain a line. Indeed, such a line would contain the point $U$ and hence would meet the conic $\gamma \cap \cC$ in a further point, which is a contradiction. Also, $\gamma \cap \cQ_i$ is a conic, $1 \le i \le q-1$, and $\bigcup_i (\gamma \cap \cQ_i)$ consists of $q^2-q+1$ points of $\gamma$. Each of the conics in $\cF'$ admits as a nucleus the same point, say $C$, where $C$ is a point of $t$ distinct from $U$, see \cite[Table 7.7]{H}. It turns out that the lines of $\gamma$ meeting each of the sets $\Sigma \setminus \cH, \cC \setminus (\pi \cup \cH), \cQ_i$, $1 \le i \le q-1$, in exactly one point are the $q-1$ lines passing through $C$ and not containing $U$ and $U'$. 
	\end{proof}

We are ready to prove our main result of this section.

\begin{theorem}\label{code1}
There exists a set $\cL$ consisting of $q^3+1$ pairwise skew lines and a set $\cP$ consisting of $q^3+1$ planes mutually intersecting in exactly a point, such that no line of $\cL$ is contained in a plane of $\cP$. 
\end{theorem}
\begin{proof}
Let $\cP_1$ be a good plane--orbit. Let $\cP_2$ be the set of $q+1$ planes generated by a line of $\cR_1$ and the point $N$. Of course, two distinct planes in $\cP_2$ share exactly the point $N$. We claim that a plane of $\cP_1$, say $\gamma$, and a plane of $\cP_2$, say $\tau$, meet in a point.

 Indeed, if $t = \gamma \cap \Sigma$ is disjoint from $\ell' = \tau \cap \Sigma$ or $\ell' = \ell$, then the assertion is trivial. If $t \cap \ell'$ is the point $U$, with $\ell' \neq \ell$, then let $R$ be the point in common between the unique line in $\cR_2$ containing $U$ and $\ell$. Let $U' = \gamma \cap \pi$. By Remark \ref{description}, we have that $U' \notin R N$ and, hence, $\langle \gamma, \tau \rangle$ is the whole hyperplane $S$. Hence, $\gamma \cap \tau$ is a point and $\cP = \cP_1 \cup \cP_2$ is a set of $q^3+1$ planes mutually intersecting in exactly a point. From Proposition \ref{pencil}, among the $q^2-q$ good line--orbits, there are exactly $q-1$ of them covered by members of $\cP_1$. Hence, there exists a good line--orbit, say $\cL_1$, such that no line of $\cL_1$ is contained in a plane of $\cP_1$. Notice that the $(q+1)(q^3-q)$ points covered by the lines of $\cL_1$ are those in $S \setminus (\pi \cup \cH)$. Therefore, a line in $\cL_1$ cannot be contained in a plane of $\cP_2$. Moreover, if $\cL= \cL_1 \cup \cR_2$, we have that $\cL$ is a set of size $q^3+1$ consisting of mutually disjoint lines. In particular, no line of $\cL$ is contained in a plane of $\cP$, as required.    
\end{proof}

\begin{remark}
With the same notation as in Theorem \ref{code1}, the code $\cC = \cP_1 \cup \cP_2 \cup \cL_1 \cup \cR_2$ is an optimal code of type $IV)$. In the same fashion, it can be seen that, if $\pi'$ is a plane containing $\ell$ and not contained in $\Sigma$, then the code $\cC = \cP_1 \cup \left(\cP_2 \setminus \{ \pi \}\right) \cup \{ \pi' \} \cup \cL_1 \cup \cR_2$ is an optimal code of type $IV)$. Let $\pi' \ne \pi$ and let $S_1$ be the solid generated by $\pi$ and $\pi'$. If $S_2$ is a solid containing $\pi$ and distinct from $S_1$ and $r_2$ is the line of $\cR_2$ contained in $S_2$, we have that the code $\cC = \cP_1 \cup \left(\cP_2 \setminus \{ \pi \}\right) \cup \{ \pi' \} \cup \cL_1 \cup \left(\cR_2 \setminus \{ r_2 \}\right) \cup \{ S_2 \}$ is an optimal code of type $II)$. On the other hand, the code $\cC = \cP_1 \cup \left(\cP_2 \setminus \{ \pi \}\right) \cup \cL_1 \cup \left(\cR_2 \setminus \{ r_2 \}\right) \cup \{ S_2, r_2 \cap \ell \}$ is an optimal code of type $III)$. Let $\cP_3$ be the set of $q+1$ planes generated by a line of $\cR_2$ and the point $N$. Then, arguing as in the proof of Theorem \ref{code1}, it can be proved that $\cC = \cP_1 \cup \cP_3 \cup \cL_1 \cup \cR_1$ is an optimal code of type $IV)$. Let $r$ be a line of $\pi$ distinct from $\ell$ and not containing $N$ and let $V$ be the point $\ell \cap r$. Let $\tau$ be a plane of $\cP_3$ not containing $V$ and let $V'$ be the point $\tau \cap \ell$. Then the code $\cC = \cP_1 \cup \left(\cP_3 \setminus \{ \tau \}\right) \cup \cL_1 \cup \left(\cR_1 \setminus \{ \ell \}\right) \cup \{ r, V' \}$ is an optimal code of type $I)$.   
\end{remark}

%\bigskip
%\footnotesize
\noindent{\bf Acknowledgements.}
The authors wish to thank the Ghent University student Jozefien D'haeseleer for giving the permission to use the drawings  that she made when studying this article within the framework of her master project \cite{JD:17}.

\end{document}